\documentclass{article}
\usepackage[english]{babel}
\usepackage[utf8]{inputenc}
\usepackage{amsmath}
\usepackage{graphicx}
\usepackage{amsfonts}
\usepackage{enumitem}
\usepackage{amssymb}
\usepackage[colorinlistoftodos]{todonotes}
\usepackage{geometry}
\usepackage{amsthm}
\usepackage{comment}
\usepackage{hyperref}
\usepackage{tikz}
\usetikzlibrary{math, calc}

\newcommand{\abs}[1]{\left\lvert#1\right\rvert}

\newcommand{\floor}[1]{\left\lfloor #1 \right\rfloor}
\newcommand{\ceil}[1]{\left\lceil #1 \right\rceil}
\newcommand{\paren}[1]{\left( #1 \right)}

\newcommand{\set}[1]{\left\{ #1 \right\}}
\newcommand{\setcond}[2]{\left\{ #1 \;\middle\vert\; #2 \right\}}

\newcommand{\EE}{\mathbb{E}}

\newcommand{\RR}{\mathbb{R}}
\newcommand{\PP}{\mathbb{P}}
\newcommand{\NN}{\mathbb{N}}
\newcommand{\ZZ}{\mathbb{Z}}

\newcommand{\cA}{\mathcal{A}}
\newcommand{\cB}{\mathcal{B}}
\newcommand{\cC}{\mathcal{C}}
\newcommand{\cD}{\mathcal{D}}
\newcommand{\cE}{\mathcal{E}}

\newcommand{\cP}{\mathcal{P}}

\newcommand{\ones}{\mathbf{1}}

\newtheorem{thm}{Theorem}[section]
\newtheorem{lem}[thm]{Lemma}
\newtheorem{cor}[thm]{Corollary}
\newtheorem{conj}[thm]{Conjecture}
\newtheorem{claim}[thm]{Claim}

\newtheorem{prop}[thm]{Proposition}
\newtheorem{prob}[thm]{Problem}
\newtheorem{fact}[thm]{Fact}
\newtheorem*{rmk}{Remark}

\theoremstyle{definition}
\newtheorem{defn}[thm]{Definition}

\DeclareMathOperator{\Vol}{Vol}
\DeclareMathOperator{\Var}{Var}

\title{On the largest sum-free subset of the lattice cube}
\author{Peter Keevash\thanks{Mathematical Institute, University of Oxford, UK. Email: {\tt keevash@maths.ox.ac.uk}. Research supported by ERC Advanced Grant 883810.} \and Jeck Lim\thanks{Mathematical Institute, University of Oxford, UK. Email: {\tt jeck.lim@maths.ox.ac.uk}. Research supported by ERC Advanced Grant 883810.}}
\date{}

\begin{document}
\maketitle

\begin{abstract}
We determine the limiting density of the largest sum-free subset of the lattice cube $\{1,2,\ldots,n\}^d$ for all $d$, thus resolving the natural conjecture that it is constructed by two appropriate hyperplane slices. Equivalently, we show that the largest measure of a sum-free subset of the hypercube $(0,1)^d\subset \mathbb{R}^d$ is attained by $\setcond{x\in (0,1)^d}{1\leq L(x)<2}$ for some linear map $L:\mathbb{R}^d\to \mathbb{R}$. It is natural to conjecture that the same phenomenon might hold if one replaces the hypercube by any convex set not containing the origin, but we give an example to show that for sufficiently large $d$ this is not the case.
\end{abstract}

\section{Introduction}

A subset $S$ of an abelian group is called \emph{sum-free} if there are no $x,y,z\in S$ such that $x+y=z$. The aim of this paper is to determine the density of the largest sum-free subset of the lattice cube $\{1,2,\ldots,n\}^d$. For notational convenience, we denote by $[n]$ the set $\{0,1,\ldots,n-1\}$, and we will be studying sum-free subsets of $[n+1]^d$. If $S\subset [n+1]^d$ is the largest sum-free subset, then we would like to determine the value
\[c_d := \lim_{n\to\infty} \frac{|S|}{n^d}.\]
For $d=1$, it is folklore that $c_1=1/2$. The problem of determining $c_d$ for $d\geq 2$ was posed by Aydinian~\cite[Problem 10]{C05} (in the case $d=2$) and Cameron~\cite{C02}, and is also featured in a list of 100 open problems by Green~\cite[Problem~6]{G100}. The case $d=2$ was resolved by Elsholtz and Rackham~\cite{ER17}, with $c_2=3/5$. Recently, Lepsveridze and Sun~\cite{LS23} resolved the cases $d=3$ and $d=4$.

In general, it is believed that the slice
\[\setcond{x\in [n+1]^d}{u\leq \sum_i x_i<2u}\]
is optimal, where $u$ is chosen to maximize the size of the slice. To quantify $u$, it is convenient to consider the continuous version in the limit $n\to\infty$.

Let $\ones_d$ denote the all ones vector in $\RR^d$. If the context is clear, we just write $\ones$. For $u\in [0,d]$, let $S_d(u) := \setcond{x\in [0,1]^d}{u \leq \sum_i x_i < 2u}$. Let $u_d^*$ be the value of $u$ that maximizes $\Vol(S_d(u))$, and let $c_d^*$ be the maximum volume $\Vol(S_d(u_d^*))$. 

Note that the set $\setcond{x\in [n+1]^d}{u_d^*n \leq \sum_i x_i<2u_d^*n}$ is sum-free and has asymptotic density $c_d^*$, thus we have the easy bound $c_d\geq c_d^*$. The conjecture is that $c_d=c_d^*$ for all $d$, which was known up to $d\leq 4$. We prove this conjecture for all $d$.

\begin{thm} \label{thm:main}
    Let $d\geq 1$ and $S\subseteq [n+1]^d$ be sum-free. Then $|S|\leq c_d^* n^d + O(n^{d-1})$.
\end{thm}

We make some remarks on the error term. For $d=1$, it is folklore that the maximum possible size of $S$ is $\ceil{n/2}$. For $d=2$, \cite{ER17} gives an error term of $O(n)$. For $d=3,4$, while \cite{LS23} obtained the optimal leading constant, their error term is only $O(n^{d-1/2})$. Hence, we need to reprove their result for $d=3,4$ to obtain our desired error term of $O(n^{d-1})$.

Our proof follows the same approach as Lepsveridze and Sun~\cite{LS23}. Their idea begins with a simple but key inequality relating the sizes of one-dimensional fibers of $S$, then rewrites the problem of bounding $|S|$ as a linear program. It turns out that this linear program is sufficient to obtain the optimal bound for $|S|$. To solve this linear program (or rather, to obtain a good upper bound), it is enough to construct feasible weights for the dual linear program. The construction of such weights essentially amounts to showing that the uniform distributions on certain slices of the cube are jointly mixable.

\begin{defn}
Let $\mu_1,\mu_2,\mu_3$ be distributions on $\RR^d$. We say that $(\mu_1,\mu_2,\mu_3)$ is \emph{jointly mixable} if there exists a coupling $(X_1,X_2,X_3)$ with marginals $X_i\sim \mu_i$ for all $i$ such that $X_1+X_2+X_3=\EE[\mu_1]+\EE[\mu_2]+\EE[\mu_3]$ almost surely.
\end{defn}

\begin{defn}
    For $r\in \ZZ$, let $L_{d}(r) := \setcond{x\in [n+1]^d}{\sum_i x_i = r}$. For $r\in \RR$, let $L_d^\RR(r) := \setcond{x\in [0,1]^d}{\sum_i x_i = r}$. For any set $L$, finite or of non-zero finite measure, denote by $U(L)$ the uniform distribution on $L$.
\end{defn}

Lepsveridze and Sun reduced Theorem~\ref{thm:main} to the following joint mixability result (and another weaker but more complicated coupling which we ultimately avoid).
\begin{conj}
    For $d\geq 2$, the triple
    \[U(L_d^\RR(u_d^*)),\ U(L_d^\RR(u_d^*)),\ U(L_d^\RR(d-2u_d^*))\]
    is jointly mixable.
\end{conj}
They are able to explicitly construct such a coupling for $d=3,4$, then obtain a discretised version of the coupling.

The main contribution of this paper is the following general joint mixability result in the discrete setting, which implies the above conjecture by taking suitable limits.

\begin{thm} \label{thm:joint-mix}
    Let $d\geq 1$ and $r,s,t$ be nonnegative integers with $r+s+t=dn$. Then the triple 
    \[U(L_{d}(r)),\ U(L_{d}(s)),\ U(L_{d}(t))\]
    is jointly mixable.
\end{thm}

We will also briefly discuss the problem of determining the size of the largest sum-free subset of convex sets besides the hypercube $[0,1]^d$.

\subsection{Acknowledgements}
ChatGPT-5.4 provided the main ideas used in the proof of Theorem~\ref{thm:joint-mix}, and helped generate the code for numerically verifying Lemma~\ref{lem:real-layer-estimate} for small $d$.

\section{Joint mixability}

We will prove Theorem~\ref{thm:joint-mix} by induction on $d$. There is nothing to prove for $d=1$, since the sets $L_{1}(r),L_{1}(s),L_{1}(t)$ are singletons. We first prove the base case $d=2$, which is equivalent to the following result by Perchet, Rigollet, and Le Gouic~\cite{PRL24}. The continuous version was also proved by Wang and Wang~\cite{WW16}.

\begin{lem}[{\cite[Proposition~15]{PRL24}}] \label{lem:1d-joint-mix}
    Let $r,s,t$ be nonnegative integers. Then $(U([r+1]),U([s+1]),U([t+1]))$ is jointly mixable if and only if $r+s+t$ is even and $r,s,t$ satisfy the triangle inequality, i.e. $r+s \geq t$, $s+t \geq r$, and $t+r \geq s$.
\end{lem}

\begin{proof}[Proof of Theorem~\ref{thm:joint-mix} for $d=2$]
    By projecting onto the first coordinate, the joint mixability of $(U(L_{d}(r)),U(L_{d}(s)),U(L_{d}(t)))$ is equivalent to the joint mixability of intervals
    \[[\max(0,r-n), \min(n,r)], [\max(0,s-n), \min(n,s)], [\max(0,t-n), \min(n,t)].\]
    By translating the intervals so that they start at zero and applying Lemma~\ref{lem:1d-joint-mix}, we see that the joint mixability of the above intervals is equivalent to the numbers
    \[\min(n,r)-\max(0,r-n), \min(n,s)-\max(0,s-n), \min(n,t)-\max(0,t-n)\]
    satisfying the triangle inequality and having even sum. The even sum is easy to verify since $\min(n,r)-\max(0,r-n)\equiv r\pmod{2}$ and $r+s+t=2n$, so it suffices to check the triangle inequality. By symmetry, it suffices to show that
    \begin{equation} \label{eq:rst-triangle}
        \min(n,r)-\max(0,r-n)+\min(n,s)-\max(0,s-n)\geq \min(n,t)-\max(0,t-n).
    \end{equation}
    Recall that $r+s+t=2n$. If $r,s\leq n$, then \eqref{eq:rst-triangle} is equivalent to $r+s\geq \min(n,2n-r-s)-\max(0,n-r-s)$, which is true since $r+s=n-(n-r-s)\geq \min(n,2n-r-s)-\max(0,n-r-s)$. If one of $r,s$ is greater than $n$, say $r>n$, then since $r+s+t=2n$, we have $s,t<n$. Hence, \eqref{eq:rst-triangle} is equivalent to $2n-r+s\geq t=2n-r-s$, which is true.
\end{proof}

We now assume $d\geq 3$. To find a coupling $(X,Y,Z)$ on $L_{d}(r)\times L_{d}(s)\times L_{d}(t)$ with $X+Y+Z=n\ones$, we will find a suitable coupling $(X_1,Y_1,Z_1)$ on the first coordinates, then use the induction hypothesis to couple the remaining coordinates. In order for this to yield a uniform marginal on $L_{d}(r)$, the law of $X_1$ should be the law of the first coordinate of $U(L_{d}(r))$. The coupling $(X_1,Y_1,Z_1)$ is obtained by setting up certain linear equations, then applying Brouwer's fixed-point theorem (or just the Perron--Frobenius theorem) to find a solution.

For $p=(i,j,k)\in [n+1]^3$, let 
\[r_p := r-i,\quad s_p := s-j,\quad t_p := t-k.\]
The support of the coupling on the first coordinates is
\[A_{d}(r,s,t) := \setcond{p=(i,j,k)\in [n+1]^3}{i+j+k=n,\ 0\leq r_p,s_p,t_p\leq (d-1)n}.\]
It is straightforward to check that $A_{d}(r,s,t)$ is non-empty. By our induction hypothesis, for each $p\in A_{d}(r,s,t)$, there exists a coupling $(\tilde{X},\tilde{Y},\tilde{Z})$ on $L_{d-1}(r_p)\times L_{d-1}(s_p)\times L_{d-1}(t_p)$ with $\tilde{X}+\tilde{Y}+\tilde{Z}=n\ones_{d-1}$. Take one such coupling, and let $\nu_p$ be the law of the first-coordinate triple $(\tilde{X}_1,\tilde{Y}_1,\tilde{Z}_1)$. 
\begin{lem} \label{lem:nu-support}
    For $p\in A_{d}(r,s,t)$, the measure $\nu_p$ is supported on $A_{d}(r,s,t)$.
\end{lem}
\begin{proof}
    Let $(x,y,z)$ be in the support of $(\tilde{X},\tilde{Y},\tilde{Z})$, then the support of $\nu_p$ is the set of possible values of $(x_1,y_1,z_1)$. Since $x\in L_{d-1}(r_p), y\in L_{d-1}(s_p), z\in L_{d-1}(t_p)$ and $x+y+z=n\ones_{d-1}$, we have $x_1+y_1+z_1=n$. Furthermore, $x_1\leq r_p\leq r$ and 
    \[x_1=r_p-x_2-\cdots -x_{d-1}=r-i-x_2-\cdots -x_{d-1}\geq r-(d-1)n.\]
    Thus, $0\leq r-x_1\leq (d-1)n$. Similarly, $0\leq s-y_1\leq (d-1)n$ and $0\leq t-z_1\leq (d-1)n$. Hence, $(x_1,y_1,z_1)\in A_{d}(r,s,t)$.
\end{proof}

Let $\cP(A_{d}(r,s,t))$ denote the space of probability distributions supported on $A_{d}(r,s,t)$, which can be viewed as a simplex with $|A_{d}(r,s,t)|$ vertices. Let $T:\cP(A_{d}(r,s,t))\to \cP(A_{d}(r,s,t))$ be the map given by
\[T(\kappa) := \sum_{p\in A_{d}(r,s,t)} \kappa(p) \nu_p.\]
By Lemma~\ref{lem:nu-support}, $T(\kappa)$ is a probability distribution supported on $A_{d}(r,s,t)$, so $T$ is a well-defined map. Note that $T$ is linear, hence continuous, and $\cP(A_{d}(r,s,t))$ is convex and compact. Thus, by Brouwer's fixed-point theorem, there exists a fixed point of $T$. Take a fixed point $\kappa$, so we have
\begin{equation} \label{eq:kappa}
    \kappa = \sum_{p\in A_{d}(r,s,t)} \kappa(p) \nu_p.
\end{equation}

We will show that $\kappa$ has the desired marginal distributions. Let $N_{d}(r)=|L_{d}(r)|$, so that $N_{d}(r)>0$ whenever $0\leq r\leq dn$, and $N_d(r)=0$ otherwise. Counting the size of $L_{d}(r)$ by iterating over the first coordinate, we have 
\begin{equation} \label{eq:N-sum}
    N_{d}(r) = \sum_{i=0}^n N_{d-1}(r-i).
\end{equation}
The law of the first coordinate of $U(L_{d}(r))$ is given by
\[q_{d,r}(u) := \frac{N_{d-1}(r-u)}{N_{d}(r)},\]
whenever $0\leq r\leq dn$ and $0\leq u\leq n$. For all other triples $(d,r,u)$, we set $q_{d,r}(u)=0$. 

\begin{lem} \label{lem:qdr}
    For $0\leq r\leq dn$ and every $u \in [n+1]$ we have
    \begin{equation} \label{eq:qdr}
        q_{d,r}(u) = \sum_{i=0}^n q_{d,r}(i)q_{d-1,r-i}(u).
    \end{equation}
\end{lem}
\begin{proof}
    The right-hand side is equal to
    \[\sum_i \frac{N_{d-1}(r-i)}{N_{d}(r)} \cdot \frac{N_{d-2}(r-i-u)}{N_{d-1}(r-i)} = \frac{1}{N_{d}(r)} \sum_i N_{d-2}(r-i-u),\]
    where the sum is over $i$ such that $0\leq i\leq n$ and $0\leq r-i\leq (d-1)n$. Note that if $r-i<0$ or $r-i>(d-1)n$, then $N_{d-2}(r-i-u)=0$. Thus, the right-hand side is equal to
    \[\frac{1}{N_{d}(r)} \sum_{i=0}^n N_{d-2}(r-i-u) = \frac{N_{d-1}(r-u)}{N_{d}(r)} = q_{d,r}(u),\]
    where the first equality follows from \eqref{eq:N-sum}.
\end{proof}

Let $q^*$ be the law of the first coordinate of $\kappa$, that is,
\[q^*(u) = \sum_{p=(i,j,k)\in A_{d}(r,s,t): i=u} \kappa(p).\]
Our aim is to show that $q^*=q_{d,r}$. Note that the first coordinate of $\nu_p$ has the same law as the first coordinate of $U(L_{d-1}(r-i))$, which is $q_{d-1,r-i}$. Then by \eqref{eq:kappa}, $q^*$ satisfies
\begin{align*}
    q^*(u) &= \sum_{p=(i,j,k)\in A_{d}(r,s,t)} \kappa(p) q_{d-1,r-i}(u)\\
    &= \sum_{i=0}^n \sum_{j,k: (i,j,k)\in A_{d}(r,s,t)} \kappa(i,j,k) q_{d-1,r-i}(u)\\
    &= \sum_{i=0}^n q^*(i)q_{d-1,r-i}(u).
\end{align*}
Thus, $q^*$ satisfies the same equation \eqref{eq:qdr} as $q_{d,r}$ in Lemma~\ref{lem:qdr}. To show that they must be the same, observe that \eqref{eq:qdr} says that $q_{d,r}$ is a stationary distribution of the Markov chain with transition matrix $M(i,u)=q_{d-1,r-i}(u)$. The state space of this Markov chain is the set
\[J:=\setcond{i\in \ZZ}{0\leq i\leq n, 0\leq r-i\leq (d-1)n}.\]
Note that for $i,u\in J$, we have $M(i,u)>0$ if and only if $0\leq r-i-u\leq (d-1)n$.
In particular, $M(i,u)>0$ if and only if $M(u,i)>0$, 
so we can represent connectivity in the chain by an undirected graph $G$ on $J$.
If $r\leq (d-1)n$, then $0\in J$ and $M(i,0)>0$ for all $i\in J$, so the Markov chain is ergodic, as $G$ is connected and not bipartite (due to the loop at $0$). If $r>(d-1)n$, then $n\in J$ and $M(i,n)>0$ for all $i\in J$ since $d\geq 3$, so the Markov chain is also ergodic in this case. In either case, the Markov chain is ergodic, so it has a unique stationary distribution, which must be $q_{d,r}=q^*$. Thus, the law of the first coordinate of $\kappa$ is $q_{d,r}$. Similarly, the law of the second coordinate of $\kappa$ is $q_{d,s}$ and the law of the third coordinate of $\kappa$ is $q_{d,t}$.

We are now ready to prove the joint mixability of $U(L_d(r)),U(L_d(s)),U(L_d(t))$. 

\begin{proof}[Proof of Theorem~\ref{thm:joint-mix} for $d\geq 3$]
    Let $\kappa$ be as above. Sample $(i,j,k)$ from $\kappa$, then $i$ has the law $q_{d,r}$, $j$ has the law $q_{d,s}$ and $k$ has the law $q_{d,t}$. By our induction hypothesis, there exists a coupling $(\tilde{X},\tilde{Y},\tilde{Z})$ on $L_{d-1}(r-i)\times L_{d-1}(s-j)\times L_{d-1}(t-k)$ with $\tilde{X}+\tilde{Y}+\tilde{Z}=n\ones_{d-1}$, where $\tilde{X},\tilde{Y},\tilde{Z}$ are uniform on their respective sets. Set 
    \[X=(i,\tilde{X}),\ Y=(j,\tilde{Y}),\ Z=(k,\tilde{Z}).\]
    We claim that this gives the required coupling. Indeed, $i+j+k=n$ and $\tilde{X}+\tilde{Y}+\tilde{Z}=n\ones_{d-1}$, so $X+Y+Z=n\ones_d$. It suffices to verify that $X\sim U(L_d(r))$, since the argument for $Y,Z$ is identical. By definition, the law of the first coordinate of $U(L_d(r))$ is $q_{d,r}$, which is the same as the law of $i$. Conditioned on $i$, the law of the remaining coordinates of $X$ is uniform on $L_{d-1}(r-i)$, which is the same for $U(L_d(r))$. Therefore, $X$ is uniform on $L_d(r)$.
\end{proof}

\begin{cor} \label{cor:joint-mix-sum}
    For $d\geq 1$ and $r,s$ nonnegative integers with $r+s\leq dn$, there exists a coupling $(X,Y,Z)$ on $L_{d}(r)\times L_{d}(s)\times L_{d}(r+s)$ with $X+Y=Z$, with marginals $X\sim U(L_{d}(r))$ and $Y\sim U(L_{d}(s))$ and $Z\sim U(L_{d}(r+s))$.
\end{cor}
\begin{proof}
    This follows from Theorem~\ref{thm:joint-mix} by taking $t=dn-r-s$, and observing that if $Z'\sim U(L_{d}(t))$, then $Z=n\ones-Z'\sim U(L_{d}(r+s))$.
\end{proof}

\section{Slice estimates}

Recall that $u_d^*$ is the value of $u$ that maximizes $\Vol(S_d(u))$. Let $L_d^{\RR}(r) := \setcond{x\in [0,1]^d}{\sum_i x_i = r}$, the $r$-slice of the hypercube $[0,1]^d$, whose volume is normalised so that $\int_0^d \Vol_{d-1}(L_d^{\RR}(r))dr=\Vol_d([0,1]^d)=1$. $u_d^*$ is a root of the derivative of $\Vol(S_d(u))$, which is
\[\frac{d}{du}\Vol(S_d(u)) = \Vol_{d-1}(L_d^{\RR}(u)) - 2\Vol_{d-1}(L_d^{\RR}(2u)).\]
Thus $\Vol_{d-1}(L_d^{\RR}(u_d^*))=2\Vol_{d-1}(L_d^{\RR}(2u_d^*))$. Note that there is a unique solution to the above equation in $(0,d/2)$. Indeed, there are no solutions for $u\in (0,d/4]$ since $\Vol_{d-1}(L_d^{\RR}(2u))>\Vol_{d-1}(L_d^{\RR}(u))$. For $u\in (d/4,d/2)$, $\Vol_{d-1}(L_d^{\RR}(u))$ is strictly increasing while $\Vol_{d-1}(L_d^{\RR}(2u))$ is strictly decreasing. Checking the endpoints, we see that there is a unique solution $u_d^*\in (d/4,d/2)$. We record the first few values of $u_d^*$ in the following table.
\begin{center}
\begin{tabular}{c|c}
$d$ & $u_d^*$ \\
\hline
1 & $1/2$ \\
2 & $4/5=0.8$ \\
3 & $\frac{15-\sqrt{15}}{10}\approx 1.112702$ \\
4 & $\approx 1.448321$ \\
5 & $\approx 1.780682$ \\
\end{tabular}
\end{center}

Let $f_d(x)=\Vol_{d-1}(L_d^{\RR}(x))$, then $f_d$ is the probability density function of the sum of $d$ independent uniform $[0,1]$ random variables. In particular, $f_d$ is supported on $[0,d]$, is symmetric about $d/2$ and has a unique maximum at $d/2$. The main estimate we need is the following.

\begin{lem} \label{lem:real-layer-estimate}
    The following inequalities hold for $d\geq 3$.
    \begin{enumerate}
        \item $f_{d-1}(2u_d^*-1)\geq 2.01f_{d-1}(u_d^*-1).$
        \item $0.99f_{d-1}(u_d^*)\geq 2f_{d-1}(u_d^*-1)+4f_{d-1}(2u_d^*).$
    \end{enumerate}
\end{lem}

\subsection{Heuristics}

Let us first give some non-rigorous heuristics justifying the above inequalities for large $d$, and also obtain estimates for $u_d^*$. Recall that $u_d^*$ is the unique solution to $f_d(u)=2f_d(2u)$, for $u\in [0,d/2]$. Since $f_d(d/3) = f_d(2d/3)$, we expect $u_d^*$ to be close to $d/3$. 

Using large deviation estimates, we can estimate the behaviour of $f_d$ near $d/3$. Let $X$ be a uniform random variable on $[0,1]$. Let $M(\lambda)=\EE[e^{\lambda X}]=\frac{e^\lambda-1}{\lambda}$ be the moment generating function of $X$, and let $K(\lambda)=\log M(\lambda)$ be the cumulant generating function of $X$. Let $K^*(x)=\sup_{\lambda\in \RR}(\lambda x - K(\lambda))$, the Fenchel--Legendre transform of $K$. Then by standard large deviation estimates (which can be proven using Cramer's theorem~\cite[Theorem~2.2.3]{DZ09} and log-concavity), for fixed $x\in [0,1]$, we have
\[\lim_{d\to \infty} \frac{1}{d}\log f_d(xd) = -K^*(x).\]
To continue with our heuristics, let us assume that 
\begin{equation} \label{eq:lde-assumption}
    \log f_d(xd)=-dK^*(x).
\end{equation}
To analyse the behaviour of $f_d$ near $d/3$, let $\sigma\approx -2.149126...$ be the unique solution to $K'(\sigma)=1/3$, then $K^*(1/3) = \sigma/3 - K(\sigma)$ and $(K^*)'(1/3)=\sigma$. Differentiating \eqref{eq:lde-assumption} with respect to $x$ at $x=1/3$, we have
\[\frac{f_d'(d/3)}{f_d(d/3)} = -\sigma,\]
that is to say, for small (constant-sized) $z$, we have $f_d(d/3+z) \approx f_d(d/3)e^{-\sigma z}$. Writing $u_d^* = d/3+\alpha$ and using $f_d(u_d^*)=2f_d(2u_d^*)$, we have
\[e^{-\sigma \alpha}\approx 2e^{2\sigma \alpha},\]
giving $\alpha\approx -\log 2/(3\sigma)\approx 0.107508$. Now we have the following estimates.
\begin{align*}
    f_{d-1}(u_d^*) &= f_{d-1}(d/3+\alpha) \approx f_{d-1}(d/3)e^{-\sigma \alpha}\approx 1.260 f_{d-1}(d/3), \\
    f_{d-1}(u_d^*-1) &= f_{d-1}(d/3+\alpha-1) \approx f_{d-1}(d/3)e^{-\sigma(\alpha-1)}\approx 0.147 f_{d-1}(d/3), \\
    f_{d-1}(2u_d^*-1) &= f_{d-1}(d/3-2\alpha) \approx f_{d-1}(d/3)e^{2\sigma \alpha}\approx 0.630 f_{d-1}(d/3), \\
    f_{d-1}(2u_d^*) &= f_{d-1}(d/3-2\alpha-1) \approx f_{d-1}(d/3)e^{\sigma (2\alpha+1)}\approx 0.0734 f_{d-1}(d/3).
\end{align*}
These estimates do indeed satisfy the lemma with plenty of room to spare. We also have the estimate $u_d^* \approx d/3 + 0.107508$.

\subsection{Log-concavity}

To make the above estimates rigorous, we will exploit the log-concavity of $f_d$ in order to prove explicit bounds on $f_d'(x)/f_d(x)$ near $d/3$.

\begin{defn}
    A distribution on $\RR$ with probability density function $f:\RR\to [0,\infty)$ is \emph{log-concave} if $\log f$ is a concave function.
\end{defn}
\begin{fact} \label{fact:log-concave-add}
    If $X,Y$ are independent log-concave random variables, then $X+Y$ is also log-concave.
\end{fact}

We will need the following bound on the location of the mode of a unimodal distribution in terms of its mean and variance.

\begin{lem} \label{lem:unimodal}
    Let $X$ be a random variable on $\RR$ with probability density function $f$. Suppose $\EE[X]=\mu, \Var(X)=\sigma^2$, and $a\in \RR$ is such that $f$ is increasing on $(-\infty,a)$ and decreasing on $(a,\infty)$. Then 
    \[|a-\mu|\leq \sqrt{3}\sigma.\]
\end{lem}
\begin{proof}
    By scaling and translating $X$, we may assume that $\mu=0$ and $\sigma=1$. By symmetry, we may further assume that $a\geq 0$, and our aim is to show that $a\leq \sqrt{3}$. Consider the magic polynomial
    \[g(x)=(x+a)(x-a/3).\]
    Then $g(x)\leq 0$ for $x\in [-a,a/3]$ and $g(x)>0$ everywhere else. The main property of $g$ is that
    \[\int_{-a}^a g(x)dx=0.\]
    Using nonnegativity and unimodality of $f$, we have
    \begin{align*}
        \EE[g(X)] &= \int_{-\infty}^\infty g(x)f(x)dx \\
        &\geq \int_{-a}^{a/3} g(x)f(x)dx + \int_{a/3}^a g(x)f(x)dx \\
        &\geq \int_{-a}^{a/3} g(x)f(a/3)dx + \int_{a/3}^a g(x)f(a/3)dx \\
        &= f(a/3) \int_{-a}^a g(x)dx = 0.
    \end{align*}
    Therefore, we have
    \begin{align*}
        0 &\leq \EE[g(X)] = \EE[X^2+(2a/3)X - a^2/3] = \sigma^2-a^2/3.
    \end{align*}
    This gives $a\leq \sqrt{3}\sigma$, as required.
\end{proof}
\begin{rmk}
    The constant $\sqrt{3}$ is tight, as witnessed by the uniform distribution on $[-1,1]$, which has mean $0$, variance $1/3$, and $a=1$.
\end{rmk}

\subsection{Proof of the slice estimates}

\begin{proof}[Proof of Lemma~\ref{lem:real-layer-estimate}]
    Let $X$ be the uniform random variable on $[0,1]$, then $f_d$ is the probability density function of the sum of $d$ independent copies of $X$. Since the probability density function of $X$ is the indicator function on $[0,1]$, which is log-concave, $f_d$ is also log-concave by Fact~\ref{fact:log-concave-add}. In particular, $f_d'(x)/f_d(x)$ is a decreasing function of $x$. Let $M(t)=\EE[e^{tX}]=\frac{e^t-1}{t}$ be the moment generating function of $X$.

    For $\sigma\in\RR$, let $X_\sigma$ be the exponential tilt of $X$, which is the random variable with probability density function
    \[f(x;\sigma) = \frac{e^{\sigma x}}{M(\sigma)}f(x).\]
    This is indeed a probability density function since
    \[\int_{-\infty}^\infty f(x;\sigma)dx = \frac{1}{M(\sigma)}\int_{-\infty}^{\infty} e^{\sigma x}f(x)dx = 1.\]
    Let us compute the mean and variance of $X_\sigma$. We have
    \begin{align*}
        \EE[X_\sigma] &= \int_{-\infty}^\infty xf(x;\sigma)dx = \frac{1}{M(\sigma)}\int_{-\infty}^\infty xe^{\sigma x}f(x)dx = \frac{M'(\sigma)}{M(\sigma)}, \\
        \EE[X_\sigma^2] &= \int_{-\infty}^\infty x^2f(x;\sigma)dx = \frac{1}{M(\sigma)}\int_{-\infty}^\infty x^2e^{\sigma x}f(x)dx = \frac{M''(\sigma)}{M(\sigma)}.
    \end{align*}
    Therefore, 
    \[\Var(X_\sigma) = \EE[X_\sigma^2] - \EE[X_\sigma]^2 = \frac{M''(\sigma)}{M(\sigma)} - \paren{\frac{M'(\sigma)}{M(\sigma)}}^2.\]

    Let $Y_\sigma$ be the sum of $d$ independent copies of $X_\sigma$, and let $f_d(\cdot;\sigma)$ be the probability density function of $Y_\sigma$. Since $f(\cdot;\sigma)$ is log-concave, $f_d(\cdot;\sigma)$ is also log-concave by Fact~\ref{fact:log-concave-add}. 

    \begin{claim}
        We have
        \[f_d(x;\sigma) = \frac{e^{\sigma x}}{M(\sigma)^d}f_d(x).\]
    \end{claim}
    \begin{proof}
        We argue by induction on $d$. For $d=1$, this is true by definition. For $d\geq 2$, we have
        \begin{align*}
            f_d(x;\sigma) &= \int_{-\infty}^\infty f_{d-1}(x-y;\sigma)f(y;\sigma)dy \\
            &= \int_{-\infty}^\infty \frac{e^{\sigma (x-y)}}{M(\sigma)^{d-1}}f_{d-1}(x-y) \cdot \frac{e^{\sigma y}}{M(\sigma)}f(y) dy \\
            &= \frac{e^{\sigma x}}{M(\sigma)^d} \int_{-\infty}^\infty f_{d-1}(x-y)f(y) dy = \frac{e^{\sigma x}}{M(\sigma)^d}f_d(x).
        \end{align*}
    \end{proof}

    We would like to show that $f_d'(x)/f_d(x)$ is approximately $-\sigma^*$ for $x$ close to $d/3$, where $\sigma^*=-2.149...$ is the unique solution to $M'(\sigma^*)/M(\sigma^*)=1/3$. By log-concavity, $f_d'/f_d$ is decreasing in $x$. We will show a lower bound on $f_d'(x)/f_d(x)$ for $x$ slightly above $d/3$, and an upper bound for $x$ slightly below $d/3$. For our application, we are aiming for the lower and upper bounds of $\sigma_1:=-3$ and $\sigma_2:=-1.6$ respectively.

    Let $\mu_i=\EE[X_{\sigma_i}]$ and $s_i^2=\Var(X_{\sigma_i})$ for $i=1,2$. We have the following approximate values:
    \begin{align*}
        \mu_1 &\approx 0.280938, \quad s_1 \approx 0.236580, \\
        \mu_2 &\approx 0.372030, \quad s_2 \approx 0.271405.
    \end{align*}
    We have $\EE[Y_{\sigma_i}] = d\mu_i$ and $\Var(Y_{\sigma_i}) = d s_i^2$. Since $f_d(\cdot;\sigma_i)$ is log-concave, it is unimodal, so by Lemma~\ref{lem:unimodal}, we have $f_d'(x;\sigma_i)\leq 0$ for $x\geq d\mu_i + \sqrt{3d} s_i$, and $f_d'(x;\sigma_i)\geq 0$ for $x\leq d\mu_i - \sqrt{3d} s_i$. Therefore, for $x\geq d\mu_i + \sqrt{3d} s_i$, we have
    \begin{align*}
        0 &\geq f_d'(x;\sigma_i) = \frac{e^{\sigma_i x}}{M(\sigma_i)^d}(\sigma_i f_d(x) + f_d'(x)) \\
        \implies \frac{f_d'(x)}{f_d(x)} &\leq -\sigma_i.
    \end{align*}
    Similarly, $f_d'(x)/f_d(x)\geq -\sigma_i$ for $x\leq d\mu_i - \sqrt{3d} s_i$. Thus, we have
    \[-\sigma_2 \leq \frac{f_d'(x)}{f_d(x)} \leq -\sigma_1\]
    for $x\in [d\mu_1 + \sqrt{3d} s_1, d\mu_2 - \sqrt{3d} s_2]$. In particular, this implies that for $d\mu_1 + \sqrt{3d} s_1 \leq x\leq y\leq d\mu_2 - \sqrt{3d} s_2$, we have
    \begin{equation} \label{eq:fd-ratio-bound}
        e^{-\sigma_2(y-x)} \leq \frac{f_d(y)}{f_d(x)} \leq e^{-\sigma_1(y-x)}.
    \end{equation}
    For $d\geq 200$, the interval $[d\mu_1 + \sqrt{3d} s_1, d\mu_2 - \sqrt{3d} s_2]$ contains $[d/3-1, d/3+1]$, so \eqref{eq:fd-ratio-bound} holds for all $x,y\in [d/3-1, d/3+1]$. For $5\leq d < 200$, we directly verify the lemma by numerically computing the relevant quantities, see Appendix~\ref{sec:numerical}.

    We now estimate the value of $u_d^*$. Let $u_d^* = d/3 + \alpha$, then $f_d(u_d^*)=2f_d(2u_d^*)$ implies that
    \[\frac{f_d(d/3+\alpha)}{f_d(2d/3+2\alpha)} = \frac{f_d(d/3+\alpha)}{f_d(d/3-2\alpha)} = 2.\]
    \eqref{eq:fd-ratio-bound} then gives
    \begin{align*}
        e^{-3\sigma_2\alpha} &\leq 2\leq e^{-3\sigma_1\alpha}\\
        \implies 0.077016 &\leq \alpha \leq 0.144406.
    \end{align*}
    Note that a priori \eqref{eq:fd-ratio-bound} may not apply since we do not know the value of $\alpha$ yet. However, the above estimates show that $f_d(d/3+0.5)/f_d(d/3-1)>2$. Then, the uniqueness of $\alpha$ implies that $\alpha<0.5$, so \eqref{eq:fd-ratio-bound} does apply.

    Recall that the lemma is the following two inequalities.
    \begin{enumerate}
        \item $f_{d-1}(2u_d^*-1)\geq 2.01f_{d-1}(u_d^*-1).$
        \item $0.99f_{d-1}(u_d^*)\geq 2f_{d-1}(u_d^*-1)+4f_{d-1}(2u_d^*).$
    \end{enumerate}
    In our computations below, we only evaluate $f_{d-1}(x)$ for $x\in [(d-1)/3-1, (d-1)/3+1]$, so \eqref{eq:fd-ratio-bound} holds.

    For (1), we have
    \begin{align*}
        \frac{f_{d-1}(2u_d^*-1)}{f_{d-1}(u_d^*-1)} &= \frac{f_{d-1}(2d/3+2\alpha-1)}{f_{d-1}(d/3+\alpha-1)} \\
        &= \frac{f_{d-1}(d/3-2\alpha)}{f_{d-1}(d/3+\alpha-1)} \\
        &\geq e^{-\sigma_2(1-3\alpha)}\\
        &\geq e^{-\sigma_2}/2 \approx 4.953032/2 > 2.01.
    \end{align*}

    For (2), we have
    \begin{align*}
        \frac{2f_{d-1}(u_d^*-1)}{f_{d-1}(u_d^*)} + \frac{4f_{d-1}(2u_d^*)}{f_{d-1}(u_d^*)} &= \frac{2f_{d-1}(d/3+\alpha-1)}{f_{d-1}(d/3+\alpha)} + \frac{4f_{d-1}(2d/3+2\alpha)}{f_{d-1}(d/3+\alpha)} \\
        &= \frac{2f_{d-1}(d/3+\alpha-1)}{f_{d-1}(d/3+\alpha)} + \frac{4f_{d-1}(d/3-2\alpha-1)}{f_{d-1}(d/3+\alpha)} \\
        &\leq 2e^{\sigma_2}+4e^{\sigma_2(3\alpha+1)} \\
        &\leq 0.961800 < 0.99.
    \end{align*}

\end{proof}

\section{Weights for the dual program}

Let $\pi:\RR^d\to \RR^{d-1}$ be the projection that forgets the last coordinate. For a subset $S\subseteq [n+1]^d$ and a point $x\in [n+1]^{d-1}$, define the fiber $S(x) := \setcond{y\in [n+1]}{(x,y)\in S}$. We begin with the following key observation from Lepsveridze and Sun~\cite{LS23}.

\begin{lem}[{\cite[Lemma~2.3]{LS23}}] \label{lem:fiber-bound}
    Let $S\subseteq [n+1]^d$ be sum-free. Then for every $x,y,z\in [n+1]^{d-1}$ with $x+y=z$, we have
    \begin{equation} \label{eq:fiber-bound}
        |S(x)|+|S(y)|+|S(z)|\leq 2(n+1).
    \end{equation}
\end{lem}

Finding an upper bound for the size of a sum-free set $S$ then amounts to solving the linear program
\begin{align*}
    \text{maximize } & \sum_{x\in [n+1]^{d-1}} |S(x)| \\
    \text{subject to } & |S(x)|+|S(y)|+|S(z)|\leq 2(n+1) \text{ for all } x,y,z\in [n+1]^{d-1} \text{ with } x+y=z, \\
    & 0\leq |S(x)|\leq n+1 \text{ for all } x\in [n+1]^{d-1}.
\end{align*}

An upper bound for the optimal value is given by a feasible solution to the dual linear program. We state this as follows for motivation, but it is not used in the proof: 
\begin{align*}
    \text{minimize } & 2(n+1) \sum_{x,y,z} w(x,y,z) + (n+1) \sum_x b(x) \\
    \text{subject to } & \sum_{y,z: x+y=z} w(x,y,z) + \sum_{y,z: y+x=z} w(y,x,z) \\
    & \qquad + \sum_{y,z: y+z=x} w(y,z,x) + b(x) \geq 1 \text{ for all } x\in [n+1]^{d-1}, \\
    & w(x,y,z)\geq 0 \text{ for all } x,y,z\in [n+1]^{d-1} \text{ with } x+y=z, \\
    & b(x)\geq 0 \text{ for all } x\in [n+1]^{d-1}.
\end{align*}

As observed by Lepsveridze and Sun, the existence of joint mixability of the distributions 
\[U(L_{d}^\RR(u_d^*)),\ U(L_{d}^\RR(u_d^*)),\ U(L_{d}^\RR(d-2u_d^*))\]
gives a feasible weight function $w$ for the dual program, resulting in the optimal bound for the primal program as we will see later. 

Let $m=\floor{u_d^* n}$, and consider the partition of the cube $[n+1]^{d-1}$ into the following five sets.
\begin{align*}
    \cA & := \setcond{x\in [n+1]^{d-1}}{\ones^T x \in [0,m-n)}, \\
    \cB & := \setcond{x\in [n+1]^{d-1}}{\ones^T x \in [m-n,m]}, \\
    \cC & := \setcond{x\in [n+1]^{d-1}}{\ones^T x \in (m,2m-n)}, \\
    \cD & := \setcond{x\in [n+1]^{d-1}}{\ones^T x \in [2m-n,2m]}, \\
    \cE & := \setcond{x\in [n+1]^{d-1}}{\ones^T x \in (2m,(d-1)n]}.
\end{align*}
These sets are non-empty for $d\geq 3$, with the exception of the case $d=3$, where $\cE=\emptyset$.

\begin{lem} \label{lem:layer-estimate}
    For $d\geq 3$, the following estimates hold for sufficiently large $n$.
    \begin{enumerate}
        \item Let $r\in [0,m-n)$ and $s\in [\frac{(d-1)n}{2},2m-n)$. Then $|L_{d-1}(s)|\geq 2|L_{d-1}(r)|$.
        \item Suppose $d\geq 4$. Let $r\in [0,m-n)$, $s\in (m,\frac{(d-1)n}{2}]$, and $t\in (2m,(d-1)n]$. Then $|L_{d-1}(s)|\geq 2|L_{d-1}(r)|+4|L_{d-1}(t)|$.
    \end{enumerate}
\end{lem}
\begin{proof}
    By Lemma~\ref{lem:real-layer-estimate} and unimodality of $f_{d-1}$, we have for $r\in [0,m-n)$ and $s\in [\frac{(d-1)n}{2},2m-n)$ that
    \begin{align*}
        \Vol_{d-2}(L_{d-1}^{\RR}(s/n)) &\geq \Vol_{d-2}(L_{d-1}^{\RR}(2u_d^*-1))\\
        &\geq 2.01\Vol_{d-2}(L_{d-1}^{\RR}(u_d^*-1))\\
        &\geq 2.01 \Vol_{d-2}(L_{d-1}^{\RR}(r/n)).
    \end{align*}
    Since $|L_{d-1}(x)|=n^{d-2}\Vol_{d-2}(L_{d-1}^{\RR}(x/n))+O(n^{d-3})$, this implies (1) for sufficiently large $n$. A similar argument gives (2).
\end{proof}

Let $\Omega:=\setcond{(x,y,z)\in (\cB\times \cB\times \cD)\cup (\cA\times \cC\times \cC)\cup (\cC\times \cC\times \cE)}{x+y=z}$.
\begin{lem} \label{lem:weight}
    For $d\geq 3$, there exists a weight function $w(x,y,z)$ supported on $\Omega$ satisfying the following. For $x\in [n+1]^{d-1}$, let $W(x) := \sum_{y,z} w(x,y,z) + \sum_{y,z} w(y,x,z) + \sum_{y,z} w(y,z,x)$.
    \begin{enumerate}
        \item $w(x,y,z)\geq 0$ for all $(x,y,z)\in \Omega$.
        \item For every $x\in \cC$, we have $W(x)\leq 1$.
        \item We have
        \[\sum_{x\in \cA\cup \cB\cup \cD\cup \cE} |W(x)-1| = O(n^{d-2}).\]
    \end{enumerate}
\end{lem}
\begin{proof}
    We will split $w$ into three parts $w=w_1+w_2+w_3$, where $w_1,w_2,w_3$ are supported on $\cB\times \cB\times \cD$, $\cA\times \cC\times \cC$, and $\cC\times \cC\times \cE$ respectively. Analogously define $W_1,W_2,W_3$ for $w_1,w_2,w_3$.

    \noindent\\
    \underline{Construction of $w_1$}. By Corollary~\ref{cor:joint-mix-sum}, there exists a coupling $(X_1,Y_1,Z_1)$ on $L_{d}(m)\times L_{d}(m)\times L_{d}(2m)$ with $X_1+Y_1=Z_1$, with marginals $X_1\sim U(L_{d}(m))$, $Y_1\sim U(L_{d}(m))$, and $Z_1\sim U(L_{d}(2m))$. Under the map $\pi$, we obtain a coupling $(\pi(X_1),\pi(Y_1),\pi(Z_1))$ on $\cB\times \cB\times \cD$ with $\pi(X_1)+\pi(Y_1)=\pi(Z_1)$, where the laws of $\pi(X_1)$ and $\pi(Y_1)$ are uniform on $\cB$, and the law of $\pi(Z_1)$ is uniform on $\cD$. 

    Recall that $\Vol_{d-1}(L_d^{\RR}(u_d^*))=2\Vol_{d-1}(L_d^{\RR}(2u_d^*))$. Thus, since $m=\floor{u_d^* n}$, we have 
    \[|L_{d}(m)|=n^{d-1}\Vol_{d-1}(L_d^{\RR}(u_d^*))+O(n^{d-2}),\quad |L_{d}(2m)|=n^{d-1}\Vol_{d-1}(L_d^{\RR}(2u_d^*))+O(n^{d-2}).\]
    Hence, we have $|\cB|=|L_{d}(m)|=2|L_{d}(2m)|+O(n^{d-2})=2|\cD|+O(n^{d-2})$. Define $w_1$ by
    \[w_1(x,y,z) := |\cD| \cdot \PP((\pi(X_1), \pi(Y_1), \pi(Z_1))=(x,y,z)).\]
    Since $\pi(X_1),\pi(Y_1)$ are uniformly distributed on $\cB$, we have $W_1(x) = 2|\cD|/|\cB|=1+O(n^{-1})$ for $x\in \cB$, since $|\cB|=\Theta(n^{d-1})$. Since $\pi(Z_1)$ is uniformly distributed on $\cD$, we have $W_1(x) = |\cD|/|\cD| = 1$ for $x\in \cD$. Furthermore, $W_1(x)=0$ everywhere else. 

    \noindent\\
    \underline{Construction of $w_2$}. Let $1\leq r < m-n-1$ and
    \[s=\floor{\frac{3m-n-r}{2}},\quad t=r+s.\]
    In other words, as $r$ ranges from $m-n-2$ down to $1$, $s$ ranges from $m+1$ to $m+\floor{\frac{m-n-1}{2}}$ with each value attained at most twice, and $t$ ranges from $2m-n-1$ down to $m+\floor{\frac{m-n+1}{2}}$ with each value attained at most twice. Therefore, each value in $[m+1,2m-n-1]$ is attained at most twice by $s,t$ combined. 
    
    By Corollary~\ref{cor:joint-mix-sum}, for each $r$, there exists a coupling $(X_2^{(r)},Y_2^{(r)},Z_2^{(r)})$ on $L_{d-1}(r)\times L_{d-1}(s)\times L_{d-1}(t)$ with $X_2^{(r)}+Y_2^{(r)}=Z_2^{(r)}$, with uniform marginals. Define $w_2$ by
    \[w_2(x,y,z) := |L_{d-1}(\ones^T x)|\cdot \PP((X_2^{(\ones^T x)},Y_2^{(\ones^T x)},Z_2^{(\ones^T x)})=(x,y,z)),\]
    which is supported on $\cA\times \cC\times \cC$. For $x\in \cA$ with $0< \ones^T x < m-n-1$, we have
    \[W_2(x) = |L_{d-1}(\ones^T x)| \cdot \PP(X_2^{(\ones^T x)}=x) = 1.\]
    For $x\in \cC$, $\ones^T x$ is attained by $s$ or $t$ at most twice. So there are at most two values of $r$, say $r_1,r_2$, such that $x$ is in the support of $Y_2^{(r_i)}$ or $Z_2^{(r_i)}$. For each such $r_i$, we have $\PP(Y_2^{(r_i)}=x)$ or $\PP(Z_2^{(r_i)}=x) = 1/|L_{d-1}(\ones^T x)|$. Thus, for $x\in \cC$, we have
    \[W_2(x) = \sum_i \frac{|L_{d-1}(r_i)|}{|L_{d-1}(\ones^T x)|}.\]

    \noindent\\
    \underline{Construction of $w_3$}. Note that for $d=3$, $\cE=\emptyset$ so there is nothing to construct. Otherwise, assume $d\geq 4$. This is similar to the construction of $w_2$. Let $2m+2\leq t\leq (d-1)n$ and
    \[r=\floor{\frac{t}{2}},\quad s=t-r=\ceil{\frac{t}{2}}.\]
    In other words, as $t$ ranges from $2m+2$ to $(d-1)n$, $r$ ranges from $m+1$ to $\floor{(d-1)n/2}$ with each value attained at most twice, and $s$ ranges from $m+1$ to $\ceil{(d-1)n/2}$ with each value attained at most twice. Therefore, each value in $[m+1,\ceil{(d-1)n/2}]$ is attained at most four times by $r,s$ combined. We will assume $n$ is even for simplicity.

    By Corollary~\ref{cor:joint-mix-sum}, for each $t$, there exists a coupling $(X_3^{(t)},Y_3^{(t)},Z_3^{(t)})$ on $L_{d-1}(r)\times L_{d-1}(s)\times L_{d-1}(t)$ with $X_3^{(t)}+Y_3^{(t)}=Z_3^{(t)}$, with uniform marginals. Define $w_3$ by
    \[w_3(x,y,z) := |L_{d-1}(\ones^T z)|\cdot \PP((X_3^{(\ones^T z)},Y_3^{(\ones^T z)},Z_3^{(\ones^T z)})=(x,y,z)).\]
    Split $\cC$ into two parts
    \begin{align*}
        \cC_1 & := \setcond{x\in \cC}{\ones^T x \in (m,\tfrac{(d-1)n}{2}]}, \\
        \cC_2 & := \setcond{x\in \cC}{\ones^T x \in (\tfrac{(d-1)n}{2},2m-n)}.
    \end{align*}
    Then $w_3$ is supported on $\cC_1\times \cC_1\times \cE$. For $x\in \cE$ with $\ones^T x > 2m+1$, we have
    \[W_3(x) = |L_{d-1}(\ones^T x)| \cdot \PP(Z_3^{(\ones^T x)}=x) = 1.\]
    For $x\in \cC_1$, $\ones^T x$ is attained by $r$ or $s$ at most four times, say when $t=t_1,\ldots,t_4$. Thus, for $x\in \cC_1$, we have
    \[W_3(x) = \sum_i \frac{|L_{d-1}(t_i)|}{|L_{d-1}(\ones^T x)|}.\]

    Combining the above, we have $w=w_1+w_2+w_3$ and $W=W_1+W_2+W_3$. Clearly $w$ is supported on $\Omega$ and $w(x,y,z)\geq 0$ for all $(x,y,z)\in \Omega$, giving condition (1) in the lemma.

    For $x\in \cA$ with $0<\ones^T x < m-n-1$, we have $W(x) = W_2(x) = 1$. If $\ones^T x = 0$ or $m-n-1$, then $W(x)=0$. 
    
    For $x\in \cB$, we have $W(x) = W_1(x) = 1+O(n^{-1})$. 
    
    For $x\in \cC_1$, we have
    \[W(x) = W_2(x)+W_3(x) = \frac{\sum_i |L_{d-1}(r_i)|+\sum_i |L_{d-1}(t_i)|}{|L_{d-1}(\ones^T x)|}\leq 1\]
    by Lemma~\ref{lem:layer-estimate}(2), since there are at most two different $r_i$, each in the range $(1,m-n-1)$, and there are at most four different $t_i$, each in the range $(m+1,(d-1)n]$.

    For $x\in \cC_2$, we have 
    \[W(x) = W_2(x) = \frac{\sum_i |L_{d-1}(r_i)|}{|L_{d-1}(\ones^T x)|}\leq 1\]
    by Lemma~\ref{lem:layer-estimate}(1). This verifies condition (2) of the lemma.

    For $x\in \cD$, we have $W(x) = W_1(x)=1$. 
    
    For $x\in \cE$ with $\ones^T x > 2m+1$, we have $W(x) = W_3(x) = 1$. If $\ones^T x = 2m+1$, then $W(x)=0$. Putting everything together, we have
    \begin{align*}
        \sum_{x\in \cA\cup \cB\cup \cD\cup \cE} |W(x)-1| & = O(n^{-1})|\cB| + \sum_{x:\ones^T x=0,m-n-1}1 + \sum_{x:\ones^T x=2m+1}1\\
        & = O(n^{d-2}) + |L_{d-1}(0)| + |L_{d-1}(m-n-1)| + |L_{d-1}(2m+1)|\\
        &= O(n^{d-2}),
    \end{align*}
    verifying condition (3) in the lemma as required.
\end{proof}

Let $S^* := \setcond{x\in [n+1]^d}{\ones^T x \in [u_d^* n, 2u_d^* n)}$. Then $S^*$ is sum-free, and $|S^*|=c_d^* n^d + O(n^{d-1})$. It was observed in \cite{LS23} that for $S^*$, the inequality \eqref{eq:fiber-bound} is essentially tight for every $(x,y,z)\in \Omega$. 
\begin{lem}[{\cite[Lemma~2.4]{LS23}}] \label{lem:fiber-tight}
    For $d\geq 3$ and every $(x,y,z)\in \Omega$, we have 
    \[|S^*(x)|+|S^*(y)|+|S^*(z)| \geq 2(n+1)-1.\]
    Furthermore, for every $x\in \cC$, we have $|S^*(x)|=n+1$.
\end{lem}

With our weight function $w$, the proof of Theorem~\ref{thm:main} is identical to that of \cite{LS23}, which we rephrase here.

\begin{proof}[Proof of Theorem~\ref{thm:main}]
    As noted in the introduction, the cases $d=1,2$ are already known. Hence, we assume $d\geq 3$ in the proof. Let $S\subseteq [n+1]^d$ be sum-free. By Lemma~\ref{lem:fiber-bound} and Lemma~\ref{lem:fiber-tight}, we have
    \begin{equation} \label{eq:Sxyz}
        |S(x)|+|S(y)|+|S(z)| - (|S^*(x)|+|S^*(y)|+|S^*(z)|) \leq 1
    \end{equation}
    for every $(x,y,z)\in \Omega$. Furthermore, since $S^*(x)=n+1$ for all $x\in \cC$, we have 
    \begin{equation} \label{eq:Sx}
        |S(x)| - |S^*(x)| \leq 0 \text{ for all }x\in \cC.
    \end{equation}
    Let $w$ and $W$ be as in Lemma~\ref{lem:weight}. Taking $w(x,y,z)$ times \eqref{eq:Sxyz} over all $(x,y,z)\in \Omega$ and taking $1-W(x)$ times \eqref{eq:Sx} over all $x\in \cC$ and summing, we have
    \begin{align*}
        \sum_{x\in [n+1]^{d-1}} W(x)(|S(x)|-|S^*(x)|) + \sum_{x\in \cC} (1-W(x))(|S(x)|-|S^*(x)|) &\leq \frac13 \sum_{x\in [n+1]^{d-1}} W(x)\\
        &= O(n^{d-1}).
    \end{align*}
    The left-hand side is equal to
    \[\sum_{x\in \cA\cup \cB\cup \cD\cup \cE} W(x)(|S(x)|-|S^*(x)|) + \sum_{x\in \cC} (|S(x)|-|S^*(x)|).\]
    By Lemma~\ref{lem:weight}, the first sum is at most
    \begin{align*}
        & \sum_{x\in \cA\cup \cB\cup \cD\cup \cE} (|S(x)|-|S^*(x)|) + \sum_{x\in \cA\cup \cB\cup \cD\cup \cE} |W(x)-1| \cdot \abs{|S(x)|-|S^*(x)|}\\
        &= \sum_{x\in \cA\cup \cB\cup \cD\cup \cE} (|S(x)|-|S^*(x)|) + O(n^{d-2})\cdot O(n)\\
        &= \sum_{x\in \cA\cup \cB\cup \cD\cup \cE} (|S(x)|-|S^*(x)|) + O(n^{d-1}).
    \end{align*}
    Putting everything together, we have
    \[\sum_{x\in [n+1]^{d-1}} |S(x)|-|S^*(x)| \leq O(n^{d-1}),\]
    which implies that $|S| = \sum_x |S(x)| \leq \sum_x |S^*(x)| + O(n^{d-1}) = c_d^* n^d + O(n^{d-1})$, as required.
\end{proof}

\section{Concluding remarks}

Theorem~\ref{thm:main} implies that the largest sum-free subset of the hypercube $[0,1]^d\subset \RR^d$ is of the form $\setcond{x\in [0,1]^d}{1\leq L(x)<2}$, for some linear map $L:\RR^d\to \RR$. More generally, subsets of the form
\[\setcond{x\in [0,1]^d}{\floor{L(x)}\equiv 1\pmod{3}}\]
are also sum-free. Even more generally, for any set $K\subset \RR^d$ and one-dimensional sum-free set $T\subset \RR$, the subset
\begin{equation} \label{eq:sum-free-form}
    \setcond{x\in K}{L(x)\in T}
\end{equation}
is sum-free. It is natural to ask whether such subsets are optimal for convex shapes $K$ besides the hypercube. The following example suggests that this is not the case for large $d$.

\begin{prop}
For sufficiently large $d$, there exist a bounded convex set $K\subset \RR^d$ not containing the origin and a sum-free subset $S\subseteq K$ such that for any linear map $L:\RR^d\to \RR$, we have 
\[\Vol(S) > \Vol(\setcond{x\in K}{\floor{L(x)}\equiv 1\pmod{3}}).\]
\end{prop}

\begin{proof}[Proof sketch]
    Let $\varepsilon>0$ be sufficiently small. Let $P\subset \RR^{d-1}$ be a regular $(d-1)$-dimensional simplex of volume 1, centred at the origin. Let $P_1$ be a copy of $P$ placed on the hyperplane $\set{x_1=1-\varepsilon}$, and let $P_2$ be a translate of $-P_1$, placed on the hyperplane $\set{x_1=2+\varepsilon}$. Let $K$ be the convex hull of $P_1\cup P_2$, then $K$ is a bounded convex set not containing the origin, and let $S=K\setminus 2K$, which is sum-free. 

    Let $L$ be a linear map. Observe that the sum-free subset $\setcond{x\in K}{1\leq x_1<2}$ has volume $(1-O_d(\varepsilon))\Vol(K)$. Thus, for the subset corresponding to $L$ to beat this, it suffices to consider the set $T=\setcond{x\in K}{1\leq L(x)<2}$ bounded by a single strip. Indeed, if $K$ intersects more than one of the strips in $\setcond{x\in \RR^d}{\floor{L(x)}\equiv 1\pmod{3}}$, say it intersects $\set{1\leq L(x)<2}$ and $\set{4\leq L(x)<5}$, then by the convexity of $K$, there exists a constant $c_d>0$ such that
    \[\Vol(\setcond{x\in K}{2\leq L(x)<4}) \geq c_d \Vol(\setcond{x\in K}{1\leq L(x)<2 \text{ or } 4\leq L(x)<5}).\]
    It follows then that $\Vol(\setcond{x\in K}{\floor{L(x)}\not\equiv 1\pmod{3}}) \geq c_d\Vol(\setcond{x\in K}{\floor{L(x)}\equiv 1\pmod{3}})$.
    
    We will show that $\Vol(S) > \Vol(T)$, or equivalently, $\Vol(K\setminus S) < \Vol(K\setminus T)$. On one hand, $K\setminus S=K\cap 2K$, whose volume is approximately $2\varepsilon \Vol(P\cap (-2P))$ for small $\varepsilon$. It can be shown that $\Vol(P\cap (-2P))$ is exponentially small in $d$. 
    
    On the other hand, suppose without loss of generality that $L(x)=a_1x_1+a_2x_2$, with $a_1,a_2\geq 0$. If $a_1\leq 1$, then $K\setminus T$ contains the set $\setcond{x\in K}{L(x)<1}\supseteq \setcond{x\in K}{x_1<1,x_2<0}$. The volume of the latter set is approximately $\varepsilon \Vol(\setcond{x\in P_1}{x_2<0})$. Gr\"unbaum's inequality~\cite{G60} says that 
    \[\Vol(\setcond{x\in P_1}{x_2<0})\geq \paren{\frac{d-1}{d}}^{d-1}\Vol(P_1)>c\Vol(P_1)\]
    for some absolute constant $c>0$, for any convex body $P_1\subset \RR^{d-1}$ with centre of mass at the origin. If $a_1>1$, then $K\setminus T$ contains the set $\setcond{x\in K}{L(x)\geq 2}\supseteq \setcond{x\in K}{x_1>2,x_2>0}$, whose volume is approximately $\varepsilon \Vol(\setcond{x\in P_2}{x_2>0})>c\varepsilon$. In either case, we have $\frac{1}{\varepsilon}\Vol(K\setminus T)>c$ is bounded below by an absolute constant, while $\frac{1}{\varepsilon}\Vol(K\setminus S)$ is exponentially small in $d$. Thus, $\Vol(K\setminus S) < \Vol(K\setminus T)$ for sufficiently large $d$.
\end{proof}

Nevertheless, it is possible that for $d=2$, the largest sum-free subset is of the form \eqref{eq:sum-free-form}.

\begin{prob}
Let $K\subset \RR^2$ be a bounded convex set. Is it true that the largest sum-free subset of $K$ is of the form $\setcond{x\in K}{L(x)\in T}$ for some linear map $L:\RR^2\to \RR$ and sum-free set $T\subset \RR$?
\end{prob}

Other natural directions for further research are the stability and counting versions of our main result. In one dimension this was the classical Cameron-Erd\H{o}s  \cite{cameron1990number} conjecture, solved independently by Green \cite{green2004cameron} and by Sapozhenko \cite{sapozhenko2008cameron}. In two dimensions a corresponding result was recently obtained by Ghosal \cite{ghosal2025number}. We believe that our methods can be adapted and combined with the framework of \cite{ghosal2025number} to prove corresponding results in higher dimensions, but for the sake of brevity we leave this for future research.

\bibliographystyle{plain}
\bibliography{main}

\newpage
\appendix

\section{Numerical verification for small \texorpdfstring{$d$}{d}} \label{sec:numerical}

In this section, we discuss the details of the verification of Lemma~\ref{lem:real-layer-estimate} for $3\leq d\leq 200$. First, this can be directly verified for $d=3$, noting that $f_{d-1}(2u_d^*)=0$ in this case. For $k\in \NN$ and $t\in [0,1)$, we compute $f_d(k+t)$ by the following recursive formula (see \cite{J13}), which avoids catastrophic cancellation:
\begin{align*}
    f_d(k+t) &= \frac{1}{d-1}\paren{(k+t)f_{d-1}(k+t) + (d-k+1-t)f_{d-1}(k-1+t)},
\end{align*}
where $f_d(k+t)=0$ for $k\not\in \set{0,1,\ldots,d-1}$ and $f_1(t)=1$. Since the function $u\mapsto f_d(u)-2f_d(2u)$ is increasing for $d/4\leq u\leq d/2$, we can compute an arbitrarily small interval containing the unique root $u_d^*$ by binary search. Using interval arithmetic, we compute (intervals containing) the quantities
\begin{align*}
    A &:= \frac{f_{d-1}(2u_d^*-1)}{f_{d-1}(u_d^*-1)}, \\
    B &:= \frac{f_{d-1}(u_d^*)}{2f_{d-1}(u_d^*-1)+4f_{d-1}(2u_d^*)},
\end{align*}
and verify that $A\geq 2.01$ and $B\geq 1/0.99$. In fact, our computations show that $A>4$ and $B>2$ for all $4\leq d\leq 200$. The table below shows the values of $A,B$ and $u_d^*$ for select values of $d$.
\begin{center}
\begin{tabular}{c|c|c|c}
    $d$ & $u_d^*$ & $A$ & $B$ \\\hline
      4 & 1.448321 & 5.897488 & 3.360936\\
      5 & 1.780682 & 5.548431 & 2.903447\\
      6 & 2.112678 & 5.299078 & 2.721799\\
      7 & 2.445236 & 5.110731 & 2.606439\\
      8 & 2.778030 & 4.985979 & 2.529258\\
      9 & 3.110931 & 4.894917 & 2.474296\\
     10 & 3.443917 & 4.824602 & 2.433083\\
     20 & 6.775701 & 4.536771 & 2.272687\\
     50 & 16.774782 & 4.383697 & 2.192468\\
    100 & 33.441145 & 4.335525 & 2.167912\\
    150 & 50.107710 & 4.319765 & 2.159948\\
    200 & 66.774327 & 4.311939 & 2.156006
\end{tabular}
\end{center}

Python code implementing the verification can be downloaded at

\hspace{1cm} \url{https://people.maths.ox.ac.uk/keevash/papers/sumfree_verify.py}

It takes about 25 seconds to run on a Macbook Air M4.

\end{document}